\documentclass[11pt]{article}
\usepackage{amsmath,indentfirst,enumitem,amsthm,amssymb,tikz,color,geometry,float,ragged2e,mathtools,graphicx,url,bigstrut,bm,booktabs,caption,subcaption}
\usepackage[utf8]{inputenc}
\usepackage{color}
\usepackage{amsfonts}
\usepackage{mathtools}
\usetikzlibrary{positioning,calc,arrows,automata}
\newtheorem{proposition}{PROPOSITION}

\theoremstyle{definition}

\newtheorem{example}{EXAMPLE}[section]
\theoremstyle{plain}
\newcommand{\R}{\mathbb{R}}

\newcommand{\z}{\mathbf{z}}

\newcommand{\h}{\mathbf{h}}

\newcommand{\one}{\mathbf{1}}
\newcommand{\x}{\mathbf{x}}
\newcommand{\y}{\mathbf{y}}

\author{Silvia Noschese\thanks{Dipartimento di Matematica, SAPIENZA Universit\`a di Roma,
P.le Aldo Moro 5, 00185 Roma, Italy. Email: \texttt{noschese@mat.uniroma1.it}}
\thanks{Corresponding author.}
\and
Lothar Reichel\thanks{Department of Mathematical Sciences, Kent State University, Kent,
OH 44242, USA. Email: \texttt{reichel@math.kent.edu}}
}

\title{Communication in Multiplex Transportation Networks\\
~\\
{\normalsize
Dedicated to our friend Michela Redivo-Zaglia on the occasion of her retirement.}
}
\date{}

\begin{document}
\setlength{\arrayrulewidth}{1pt}
\maketitle

\begin{abstract}
Complex networks are made up of vertices and edges. The edges, which may be directed or
undirected, are equipped with positive weights. Modeling complex systems that consist of 
different types of objects leads to multilayer networks, in which vertices in distinct 
layers represent different kinds of objects. Multiplex networks are special vertex-aligned
multilayer networks, in which vertices in distinct layers are identified with each other
and inter-layer edges connect each vertex with its copy in other layers and have a fixed 
weight $\gamma>0$ associated with the ease of communication between layers. This paper 
discusses two different approaches to analyze communication in a multiplex. One approach 
focuses on the multiplex global efficiency by using the multiplex path length matrix, the
other approach considers the multiplex total communicability. The sensitivity of both the
multiplex global efficiency and the multiplex total communicability to structural 
perturbations in the network is investigated to help to identify intra-layer edges that 
should be strengthened to enhance communicability.
\end{abstract}

\providecommand{\keywords}[1]
{
\small	
\textbf{Keywords:} #1
}
\keywords{Multiplex network, network analysis, total communicability, global efficiency,
sensitivity analysis, multiplex path length matrix}
\vskip5pt

\providecommand{\subclass}[1]
{\small
\textbf{Mathematics Subject Classification:} #1
}
\subclass{65F15, 65F50, 05C82}

\section{Introduction}
Multilayer networks arise when one seeks to model a complex system that contains
connections and objects with distinct properties; see, e.g., \cite{DSOGA,KAB}. Multiplex
networks, or briefly \emph{multiplexes}, are special multilayer networks in which vertices
in distinct layers are identified with each other, i.e., every vertex in some layer has a 
copy in all other layers and is connected to them. Connections between vertices in 
distinct layers are furnished by \emph{inter-layer edges} that connect instances of the 
same vertex in different layers; connections between vertices in the same layer are 
represented by \emph{intra-layer} edges. 

Let the multiplex have $L$ layers and let the graph for layer $\ell$ have $N$ vertices. 
This graph is represented by an adjacency matrix $A^{(\ell)}=[a_{ij}^{(\ell)}]^N_{i,j=1}$,
whose entry $a_{ij}^{(\ell)}$ is positive if there is an edge from vertex $v_i$ to vertex
$v_j$ in layer $\ell$; if there is no such edge, then $a_{ij}^{(\ell)}=0$. The graph is 
said to be \emph{undirected} if $a_{ij}^{(\ell)}=a_{ji}^{(\ell)}$ for all 
$1\leq i,j\leq N$; otherwise the graph is \emph{directed}. When $a_{ij}^{(\ell)}>0$, this
quantity is the \emph{weight} of the edge from vertex $v_i$ to vertex $v_j$ in layer $\ell$; 
we denote this intra-layer edge by $e(v_i^{\ell}\rightarrow v_j^{\ell})$. 
A graph is said to be \emph{unweighted} if all nonvanishing 
edge-weights equal one; otherwise the the graph is weighted. All matrices $A^{(\ell)}$, 
$1\leq\ell\leq L$, that make up a multiplex are of the same size and all inter-layer edges
are undirected and have the same weight $\gamma>0$.

Applications of multiplexes include modeling transportation networks that are made up of
train and bus routes, where the train routes and bus routes define intra-layer edges in 
different layers, and the train stations and bus stops define vertices with diverse 
properties. The weight of an intra-layer edge may account for the time needed to travel 
along the road or rail represented by the edge, while the  weight $\gamma$ of the inter-layer 
edges may model the average time spent transferring between a train station and an adjacent bus 
stop; see, e.g., \cite{BS}.

A multiplex with $N$ vertices $\{v_1,v_2,\ldots,v_N\}$ and $L$ layers may be represented 
by a third-order adjacency tensor $\mathcal{A}\in\R^{N\times N\times L}$ and a parameter
$\gamma$. The horizontal slices of the tensor are the adjacency matrices $A^{(\ell)}$, 
i.e., $\mathcal{A}=[a_{ij}^{(\ell)}]_{i,j=1,2,\dots,N,\,\ell=1,2,\ldots,L}$. This 
multiplex also may be represented by a supra-adjacency matrix $B\in\R^{NL\times NL}$ with 
$N\times N$ blocks, where the adjacency matrix $A^{(\ell)}$ is the $\ell$th diagonal 
block, for $\ell=1,2,\ldots,L$, and every off-diagonal block, in position $(\ell_1,\ell_2)$
for $1\leq \ell_1,\ell_2\leq L$ and $\ell_1\ne\ell_2$, equals $\gamma \, I_N$ with the 
same $\gamma>0$. The off-diagonal block in position $(\ell_1,\ell_2)$ represents the 
inter-layer connection between the layers $\ell_1$ and $\ell_2$. Thus, 
\begin{equation}\label{Bgamma}
 B=B(\gamma)={\rm blkdiag}[A^{(1)},A^{(2)},\dots,A^{(L)}]+\gamma
(\one_L\one_L^T\otimes I_N-I_{NL}),
\end{equation}
where $\otimes$ denotes the Kronecker product. Here $I_N\in\R^{N\times N}$ is the identity
matrix, and $\one_L=[1, 1, \dots, 1]^T\in\R^L$; see \cite{DSOGA}.

In a weighted multiplex, the edge-weights depend on the network model. In this paper, the
weight associated with each intra-layer edge accounts for its ``importance'', e.g., the 
number of flights along the route modeled by the edge, or the width of a highway segment 
modeled by the edge.  Then the larger the edge-weight, the easier is the communication 
between the vertices that are connected by the edge. As an example, let there be an edge 
$e(v_i^{\ell}\rightarrow v_j^{\ell})$ with weight $a_{ij}^{(\ell)}$. If this weight 
equals the number of flights from vertex $v_i$ to vertex $v_j$ offered by airline $\ell$, 
then the reciprocal weight, $1/a_{ij}^{(\ell)}$, may be considered the average wait time 
between flights with this airline along this route. Then doubling the number of flights 
along a route corresponds to halving the wait time between flights along the edge.

We will need the notions of \emph{path} and \emph{walk} in a multiplex. A walk with $k+1$ 
vertices is a sequence of vertices $v_{i_1},v_{i_2},\ldots,v_{i_{k+1}}$ and an associated 
sequence of $k$ intra-layer edges  $e(v_{i_1}^{\ell_1}\rightarrow v_{i_2}^{\ell_1}),\ldots,
e(v_{i_k}^{\ell_h}\rightarrow v_{i_{k+1}}^{\ell_h})$ connected by $h$ inter-layer edges,
with $1\leq h\leq k$. The length of the walk defined by these vertices and edges is the 
sum of the \emph{reciprocal weights} of the edges that make up the walk. 
Vertices and edges of a walk may be repeated. A path is a walk in which no vertex is repeated. 
Let there be a path 
from vertex $v_i$ to vertex $v_j$. Then the distance $d(v_i,v_j)$ from vertex $v_i$ to 
vertex $v_j$ is the length of the shortest path from vertex $v_i$ to vertex $v_j$ measured
by the sum of the reciprocal weights of the edges of the path. If the multiplex is unweighted
and $\gamma=1$, then $d(v_i,v_j)$ is the number of edges in a shortest 
path from $v_i$ to $v_j$. Note 
that $d(v_i,v_j)$ may differ from $d(v_j,v_i)$; in fact, some distances might not be 
defined.

It is of interest to determine the ease of communication between vertices in a network in
a well-defined sense. We consider two approaches in our analysis of the communication in a
multiplex:
\begin{enumerate}
\item Construct the multiplex path length matrix $P=P(\gamma)=[p_{ij}]_{i,j=1}^N$ (to be 
defined in Section \ref{sec3}) and consider the multiplex average inverse geodesic length, 
\[
e_{\cal A}(\gamma)=\frac{1}{N(N-1)}\sum_{i,j\ne i}\frac{1}{p_{ij}},
\]
which we will refer to as the {\it multiplex global efficiency}; see \cite{BBV,NR3}. In
this approach an edge $e(v_i^{\ell}\rightarrow v_j^{\ell})$ is considered important when
it is \emph{efficient} to transmit information along the edge, e.g., if several paths of short
length end at vertex $v_i$ and/or several paths of short length start at vertex $v_j$. We 
will refer to this technique as the \emph{efficiency approach}.
\item Consider the \emph{multiplex total communicability}, which is defined by
\[
tc_B(\gamma)=\one_{NL}^T\exp_0(B)\one_{NL},
\]
where $\exp_0(B)$ is the modified exponential matrix, with $\exp_0(t)=\exp(t)-1$;
see \cite{BK,BS1,EHNR}. In this approach an edge $e(v_i^{\ell}\rightarrow v_j^{\ell})$ is
considered important when it is \emph{popular} in transmitting information, i.e., when 
vertex $v_i$ has several in-edges with large weight in layer $\ell$ and/or vertex $v_j$ 
has several out-edges with large weight in layer $\ell$. This technique will be 
referred to as the \emph{popularity approach}.
\end{enumerate}

To assess the sensitivity of a measure of communication between the vertices to 
perturbations in intra-layer edge-weights, we analyze  the ``structured'' sensitivity to 
changes of the positive entries of the tensor $\mathcal{A}$ to determine which intra-layer
edges should be strengthened to enhance the global efficiency or the total communicability 
the most.

It is the purpose of the present paper to discuss and compare two approaches to 
analyze communication in a multiplex. The \emph{efficiency approach} focuses on the 
multiplex global 
efficiency by using the multiplex path length matrix. This way to analyze multiplex
networks was introduced in \cite{NR3}.  We remark that in \cite{NR3} the weight 
associated with each intra-layer edge accounts for some kind of ``distance'', e.g., the
time required to travel from one location to another, the geographic distance between the
locations associated with the vertices that are connected by the edge, or the cost of 
traversing along the edge. Consequently, in \cite{NR3} the length of a walk is defined as
the sum of the weights of the edges that make up the walk. In the present paper, the 
length of a walk instead is defined as the sum of the reciprocal weights associated with
the edges of the walk. This results in a novel derivation of the multiplex path length 
matrix, which is used in the computation of the multiplex global efficiency.
New bounds for the latter are derived.

We also consider the multiplex total communicability and approximate this quantity
by the multiplex Perron communicability, which was defined in \cite{EHNR}. The application
of this measure is referred to as the \emph{popularity approach.}
The multiplex global efficiency and the Perron communicability help us identify 
intra-layer edges that should be strengthened to enhance communicability in the network. 

We are interested in comparing the efficiency and popularity approaches. 
In particular, we would like to study how sensitive the multiplex global efficiency and the
Perron communicability are to perturbations of the multiplex network. Related 
investigations of single-layer networks are presented in \cite{DLCJNR,NR4}.

This paper is organized as follows. In Section \ref{sec2} we discuss how the sparsity 
structure of the multiplex network can be exploited for sensitivity analysis. Sections 
\ref{sec3} and \ref{sec4} focus on the efficiency and popularity approaches, respectively.
Numerical examples that compare the efficiency and popularity approaches are reported in 
Section \ref{sec5}. Section \ref{sec6} contains concluding remarks.

It is a pleasure to dedicate this paper to Michela Redivo-Zaglia. She has made 
important contributions in many areas of computational mathematics including the solution 
of linear discrete ill-posed problems, handling breakdown in the Lanczos method,
tensor and network computations, extrapolation, sequence transformation, and also 
written papers and books on the history of mathematics; see, e.g.,
\cite{BMRZ,BRZ,BRZ2,BRZ3,BRZRS,BRZS0,BRZS,CRZT}.

\section{Structured multiplex Perron sensitivity analysis} \label{sec2} 
The following notions form the basis for our sensitivity analysis of multiplex networks.
Let the matrix $A\in\R^{N\times n}$ be nonnegative and irreducible. Then it follows from
the Perron-Frobenius theory that $A$ has a unique eigenvalue $\rho>0$ of largest magnitude
(the Perron root) and that the associated right and left eigenvectors, $\x$ and $\y$, 
respectively, 
\[
A \x = \rho \x,\qquad  \y^T A = \rho \y^T,
\]
can be normalized to be of unit Euclidean norm with all components positive. They are 
referred to as Perron vectors. Let $E\in\R^{N\times N}$ be a nonnegative matrix 
of unit spectral norm, $\|E\|_2=1$. Introduce a small positive parameter $\varepsilon$
and denote the Perron root of $A+\varepsilon E$ by $\rho+\delta\rho$. Then
\[
\delta\rho=\varepsilon \frac{\y^T E \x}{\y^T\x}+{\mathcal O}(\varepsilon^2)\quad\mbox{as}
\quad \varepsilon\searrow 0
\]
and
\[
\frac{\y^T E\x}{\y^T\x}=\frac{|\y^T E \x|}{\y^T\x}\leq\frac{\|\y\|_2\| E\|_2\|\x\|_2}
{\y^T\x}=\frac{1}{\y^T\x},
\]
with equality attained when $E$ is the \emph{Wilkinson perturbation} $W_N=\y\x^T$ 
associated with $\rho$; see \cite{Wi}. The quantity $\frac{1}{\y^T\x}$ is referred to as 
the \emph{condition number} of $\rho$ and denoted by $\kappa(\rho)$. We note that the 
spectral norm may be replaced by the Frobenius norm.

Consider the cone $\mathcal{S}$ of all nonnegative matrices in $\R^{N\times N}$ with a 
given sparsity structure and let $M|_{\mathcal{S}}$ denote a matrix in $\mathcal{S}$ that
is closest to a given nonnegative matrix $M$ with respect to the Frobenius norm, i.e., 
$M|_{\mathcal{S}}$ is the projection of $M$ onto the cone $\mathcal{S}$. Then 
$M|_{\mathcal{S}}$ is obtained by setting all the positive entries of $M$ outside the 
sparsity structure $\mathcal{S}$ to zero. 

Let $E\in{\mathcal{S}}$ be a nonnegative matrix of unit Frobenius norm, $\|E\|_F=1$. Then
\[
\frac{\y^T E \x}{\y^T\x}=\frac{|\y^T E \x|}{\y^T\x}\leq\frac{\|\y\|_2\| 
\|W_N|_{\mathcal{S}}\|_F\|\x\|_2}{\y^T\x}=\frac{\|W_N|_{\mathcal{S}}\|_F}{\y^T\x},
\]
with equality for the \emph{structured analogue of the Wilkinson perturbation} $W_N$,
\[
E=\frac{ W_N|_{\mathcal{S}} }{\|W_N|_{\mathcal{S}}\|_F }.
\]
This is the worst-case perturbation for the Perron root $\rho$ induced by a unit norm 
matrix $E\in{\mathcal{S}}$; see \cite{NP}. The quantity
\begin{equation}\label{kst}
\kappa^{\rm struct}(\rho)=\frac{\|W_N|_{\mathcal{S}}\|_F}{\y^T\x}=
\kappa(\rho)\|W_N|_{\mathcal{S}}\|_F
\end{equation}
is referred to as the \emph{structured condition number} of $\rho$. It satisfies 
$\kappa^{\rm struct}(\rho)\leq \kappa(\rho)$.

Assume that the matrix $A^+:=\sum_{\ell=1}^L A^{(\ell)}$ is irreducible. Then also the 
supra-adjacency matrix $B$ is irreducible. To determine which edge-weight(s) should be 
increased to enhance communicability the most, we apply the Perron-Frobenius theory. 
\begin{enumerate}
\item As for the efficiency approach, we analyze the Wilkinson perturbation $W_{N}$
associated with the Perron root of an $N\times N$ efficiency matrix (defined in 
Subsection \ref{sub33}), projected onto the cone 
$\mathcal{S}_{A^+}\subseteq\R^{N \times N}$ of all nonnegative matrices with the same 
sparsity structure as $A^+$.
\item As for the popularity approach, we analyze the Wilkinson perturbation $W_{NL}$
associated with the Perron root of the supra-adjacency matrix, projected onto the cone 
$\mathcal{S}_{B_d} \subseteq\R^{NL \times NL}$ of all nonnegative matrices with the same 
sparsity structure as $B_d:={\rm blkdiag}[A^{(1)},\dots,A^{(L)}]$.
\end{enumerate}

\section{The efficiency approach}\label{sec3} 
This section introduces the path length matrix for multiplexes and describes
how it can be applied to determine the importance of an edge. The sensitivity 
of the edge importance to perturbations of the weights is investigated.

\subsection{The multiplex $1$-path length matrix}
To construct the multiplex path length matrix $P=P(\gamma)$ associated with the given
multiplex network, we first introduce the third-order tensor 
$\mathcal{P}=[p_{ij}^{(\ell)}]_{i,j=1,2,\dots,N,\;\ell=1,2,\dots,L}\in
\R^{N\times N\times L}$ with entries 
\[
p_{ij}^{(\ell)}=\left\{\begin{array}{cc}
0,&\mbox{if~~} i=j,\\
1/a_{ij}^{(\ell)}, & \mbox{~~~~if~~} a_{ij}^{(\ell)}> 0,\\
\infty, & \mbox{~~otherwise}.
\end{array}\right.
\]
Then define the multiplex $1$-path length matrix 
\[
P^1=[p_{ij}^1]_{i,j=1}^N, \mbox{~~with~~} 
p_{ij}^1=\min_{\ell=1,2,\dots,L} p_{ij}^{(\ell)}.
\]
The entry $p_{ij}^1$ with $i\ne j$ represents the length of the shortest path from vertex 
$v_i$ to vertex $v_j$ made up of a single intra-layer edge, or equals infinity if there is
no edge in any layer from vertex $v_i$ to vertex $v_j$. For example, let $a_{ij}^{(\ell)}$
be the number of direct flights from vertex $v_i$ to vertex $v_j$ offered by airline 
$\ell$. Its reciprocal $1/a_{ij}^{(\ell)}$ can be interpreted as the average wait time 
between these flights. The extra-diagonal entry $p_{ij}^1$ then either represents the 
average wait time for any direct flight or equals infinity if no airline offers a direct 
flight.

\subsection{Constructing the multiplex $K$-path length matrix}
We discuss how to construct the multiplex $K$-path length matrix 
$P^K=P^K(\gamma)=[p^K_{ij}]_{i,j=1}^N$, whose entry $p^K_{ij}$ with $i\ne j$ is the length
of the shortest path from vertex $v_i$ to vertex $v_j$ made up of at most $K$ intra-layer 
edges. An analogous path length matrix has previously been introduced in \cite{NR2} to
investigate single-layer networks. The diagonal entries of $P^K$ are zero by definition.
We note that the multiplex path length matrix $P=[p_{ij}]_{i,j=1}^N$ satisfies 
$P\equiv P^{N-1}(\gamma)$, because in a multiplex path, the number of intra-layer edges 
is at most $N-1$.

The construction of the multiplex $K$-path length matrix uses 
\emph{min-plus matrix multiplication}, i.e., we carry out matrix multiplication in the 
tropical algebra; see \cite{L}:
\begin{equation*}
C=A\star B: \qquad c_{ij}=\min_{h=1,2,\ldots,N} \{a_{ih}+b_{hj}\},\qquad 1\leq i,j\leq N,
\end{equation*}
where $A=[a_{ij}]_{i,j=1}^N$, $B=[b_{ij}]_{i,j=1}^N$, and $C=[c_{ij}]_{i,j=1}^N$ are real
$N\times N$ matrices.

The multiplex $K$-path length matrix $P^K=[p_{ij}^K]_{i,j=1}^N$ can be constructed by 
means of min-plus powers of $P^1$. For single-layer networks, i.e., when $L=1$, the matrix
$P^K=[p^K_{ij}]_{i,j=1}^N$ is for $1<K\leq N-1$ given by
\[
p^{K}_{ij}=\min_{h=1,2,\ldots,N} 
\{p^{K-1}_{ih}+p^{1}_{hj}\}, \,\,\mbox {if}\,\, i\ne j,
\]
and $p^{K}_{ij}=0$ otherwise; see \cite{NR2}. When determining the entry $p_{ij}^{K}$ for
a multiplex, one has to include the cost $1/\gamma$ for each layer switch in the sum of 
the reciprocal weights of intra-layer edges of a path, because all intra-layer edges of a
shortest path of a multiplex do not necessarily belong to the same layer. In order to
see if such switching cost is relevant, one takes into account the layer of the last edge 
(i.e., the intra-layer edge from the penultimate vertex to the last vertex) of all shortest 
paths from vertex $v_i$ to vertex $v_h$ made up of at most $K-1$ edges (in case
$0<p^{K-1}_{ih}<\infty$). Only if there exists an edge from vertex $v_h$ to vertex $v_j$  
in layer $\ell$ (i.e., if  $0<p^{(\ell)}_{hj}<\infty$) and such layer is different from all the 
above mentioned layers, then the cost $1/\gamma$ is included in the 
computation of the length of the relevant path made up of at most $K$ intra-layer edges,
because there is a layer-switch preceding the last intra-layer edge.

\noindent
In this way the off-diagonal entries of the multiplex $K$-path length matrix 
$P^{K}=[p^{K}_{ij}]_{i,j=1}^N$, for $1<K\leq N-1$, are computed according to 
\[
{ p^{K}_{ij} =p^{K-1}_{i\bar{h}}+p^{(\bar{\ell})}_{\bar{h}j}+
\frac{1}{\gamma}  
\delta^{(\bar{\ell})}_{\bar{h}} 
~~\mbox {where}~~
(\bar{h},\bar{\ell})={\arg\min}_{h,\ell} \,
\{p^{K-1}_{ih}+p^{(\ell)}_{hj}+\frac{1}{\gamma} \delta^{(\ell)}_{h} } \},
\]
with $\delta^{(\ell)}_{h}=0$, if one of the following conditions holds:
\begin{itemize}
\item[ ] $p^{K-1}_{ih}=0$, i.e., $v_i=v_h$; 
\item[ ] $p^{K-1}_{ih}=\infty$, i.e., there is no path from vertex $v_i$ to vertex $v_h$
made up of at most $K-1$ edges;
\item[ ] $p^{(\ell)}_{hj}=0$, for all $\ell=1,2,\dots,L$, i.e., $v_h=v_j$; 
\item[ ] $p^{(\ell)}_{hj}=\infty$, for all $\ell=1,2,\dots,L$, i.e., there are no 
intra-layer edges from vertex $v_h$ to vertex $v_j$;
\item[ ] the intra-layer edge from vertex $v_h$ to vertex $v_j$ with weight 
$p^{({\ell})}_{hj}$ belongs to the same layer $\ell$ as the last edge of a shortest path 
made up of at most $K-1$ edges from vertex $v_i$ to vertex $v_h$ of length $p^{K-1}_{ih}$,
\end{itemize}
and $\delta^{(\ell)}_{h}=1$ otherwise; see \cite{NR3}.

\subsubsection{Min-plus powers versus powers}
Consider for ease of discussion an undirected and unweighted single-layer network, i.e., a 
simple graph. Let $A=[a_{ij}]_{i,j=1}^N$ be the adjacency matrix for the graph and define 
its $h$th power $A^h=[a^{(h)}_{ij}]_{i,j=1}^N$. The entry $a^{(h)}_{ij}$ counts the number
of walks of length $h$ between the vertices $v_i$ and $v_j$. Estrada and 
Rodriguez-Velazquez \cite{ER} defined the communicability between the vertices $v_i$ and
$v_j$ for $i\ne j$ as the $(i,j)$th entry of the matrix
\begin{equation}\label{exp0sum}
\exp_0(A)=\sum_{h=1}^\infty \frac{A^h}{h!}.
\end{equation}
The rapid growth of the denominator with $h$ ensures that the expansion converges and that
terms $\frac{A^h}{h!}$ with $h$ large contribute only little to $\exp_0(A)$. This is in 
agreement with the intuition that messages propagate better along short walks than along 
long ones. The $(i,i)$th entry of the sum \eqref{exp0sum} is commonly referred to as the 
subgraph centrality of vertex $v_i$; see \cite{ER}\footnote{In \cite{ER} $\exp(A)$ is used
instead of $\exp_0(A)$, but the term $I_N$ has no natural interpretation in network 
modeling.}.

We turn to the min-plus power $P^{K}$. The $(i,j)$th entry of $P^{K}$ gives the length of 
the shortest path between the vertices $v_i$ and $v_j$ made up of at most $K$ edges, i.e.,
it counts the number of edges of such shortest paths (since the graph is unweighted).
Compare with the $K$th partial sum $\sum_{h=1}^K A^h/h!$, whose $(i,j)$th entry is related
to the number of walks between $v_i$ and $v_j$ made up of at most $K$ edges. Thus, when 
considering the path length matrix $P$ instead of $\exp_0(A)$, one emphasizes the 
availability of a short shortest path more than the availability of several paths that can
be used for communication, thus focusing more on efficiency than popularity.

\subsection{Estimating the multiplex global efficiency}\label{sub33}
For $1\leq K\leq N-1$, we introduce the \emph{multiplex $K$-efficiency matrix} 
$P^{K}_{-1}=[p_{ij}^{K,-1}]_{i,j=1}^N$. It is obtained by replacing the off-diagonal 
entries of the $K$-path length matrix $P^K$ by their reciprocals, i.e.,
\[
p_{ij}^{K,-1}=1/p_{ij}^{K}, \quad 1\leq i,j\leq N, \quad i\ne j.
\]
Moreover, we define the \emph{multiplex global $K$-efficiency}
\[
e_{\mathcal{A}}^{K}(\gamma)=\frac{1}{N(N-1)}\sum_{i,j\ne i} \frac{1}{p_{ij}^{K}}=
\frac{1}{N(N-1)}\one_{N}^T P^{K}_{-1}\one_{N},
\]
with $1/\infty$ identified with $0$; see \cite{NR3}. 

For the multiplex global efficiency $e_{\mathcal{A}}(\gamma)$, one has 
$$
e_{\mathcal{A}}(\gamma)=\frac{1}{N(N-1)}
\sum_{i,j\ne i} \frac{1}{p_{ij}}=\frac{1}{N(N-1)}\one_{N}^T P_{-1}\one_{N},
$$
where $P_{-1}$ will be referred to as the \emph{multiplex efficiency matrix}. The 
following result shows how we can estimate the multiplex global efficiency.

\begin{proposition}
The multiplex global $K$-efficiency, for $1\leq K\leq N-1$, satisfies the inequality
$$
e_{\mathcal{A}}^{1}(\gamma)\leq e_{\mathcal{A}}^{2}(\gamma)\leq\dots\leq 
e_{\mathcal{A}}^{N-1}=e_{\mathcal{A}}(\gamma).
$$
\end{proposition}

\begin{proof}
The proof follows by observing that by construction $P_{-1}^K\leq P_{-1}^{K+1}$, for 
$1\leq K<N-1$, and that $P_{-1}\equiv P_{-1}^{N-1}$.
\end{proof}

\subsection{Increasing the multiplex efficiency}
We would like to identify the intra-layer edges that should be strengthened to enhance the
multiplex global efficiency the most and will refer to these edges as ``efficient''. 
Consider the nonnegative vectors $\h_{\rm in}\in\R^N$ and $\h_{\rm out}\in\R^N$ defined 
as follows: the $i$th entry of the former is the harmonic ${\rm in}$-centrality 
$\sum_{h\ne i} 1/p_{hi}$ of vertex $v_i$ and the $i$th entry of the latter is its harmonic
${\rm out}$-centrality $\sum_{h\ne i} 1/p_{ih}$; see  \cite{BBV}. The following  results 
hold.

\begin{proposition}\label{gl_harm}
For the multiplex global efficiency and the harmonic in- and out-centralities, we have
$$
e_{\mathcal{A}}(\gamma)=\frac{1}{N(N-1)}\|\h_{\rm in}\|_1=
\frac{1}{N(N-1)}\|\h_{\rm out}\|_1.
$$
\end{proposition}

\begin{proof}
Since all the entries of the multiplex efficiency matrix $P_{-1}$ are nonnegative 
quantities, one has that every entry of both harmonic in- and out-centrality vectors 
equals its modulus. Thus, the equalities 
$\|\h_{\rm in}\|_1=\|\h_{\rm out}\|_1=\one_{N}^T P_{-1}\one_{N}$ give the desired 
equality.
\end{proof}

\begin{proposition}\label{rho_harm}
For the Perron root of the multiplex efficiency matrix and the harmonic in- and 
out-centrality vectors, one has
$$
\rho\leq\min\{ \|\h_{\rm in}\|_{\infty}, \|\h_{\rm out}\|_{\infty}\}.
$$
\end{proposition}
\begin{proof}
One can see that $\|P_{-1}\|_1=\|\h_{\rm in}\|_{\infty}$ and 
$\|P_{-1}\|_{\infty}=\|\h_{\rm out}\|_{\infty}$. Since the spectral radius is less than or
equal to any natural matrix norm, it follows that $\rho\leq \|P_{-1}\|_1$ and 
$\rho\leq \|P_{-1}\|_{\infty}$. This yields the desired inequality. 
\end{proof}

\begin{proposition}\label{rho_gl}
For the Perron root of the multiplex efficiency matrix and the multiplex global 
efficiency, one has
$$
\rho\leq N(N-1)e_{\mathcal{A}}(\gamma).
$$
\end{proposition}

\begin{proof}
By Proposition \ref{rho_harm}, both $\rho\leq \|\h_{\rm in}\|_1$ and 
$\rho\leq \|\h_{\rm out}\|_1$ hold, having observed that for any vector $\z$ one has 
$\|\z\|_{\infty}\leq\|\z\|_1$. The proof now follows by Proposition \ref{gl_harm}.
\end{proof}

The above results lead us to expect that the multiplex global efficiency increases the 
most by increasing the edge-weights that make $\rho$ increase the most. We therefore 
should strengthen the existing edges $e(v_i^{\ell}\rightarrow v_j^{\ell})$, for 
$\ell = 1,2 \dots, L$, determined by the $(i,j)$th entry of 
$A^+=\sum_{\ell=1}^L A^{(\ell)}$ that corresponds to a largest entry of the Wilkinson 
perturbation $W_{N}$ associated with $\rho$, projected onto the cone 
${\mathcal S_{A^+}}$.
 
Since multiplexes typically have a large number of vertices, we focus on techniques that 
are well suited for large-scale networks. In case the computation of the path length 
matrix $P\equiv P^{N-1}$ is too expensive to be attractive, one may instead consider the
Perron root $\rho_K$ of the multiplex $K$-efficiency matrix $P^{K}_{-1}$ and the 
associated Perron vectors $\mathbf{x}_K$ and $\mathbf{y}_K$, and determine the Wilkinson 
matrix $W^K_{N}=\mathbf{y}_K\mathbf{x}_K^T$, for $1\leq K\leq N-1$ large enough, to 
identify edges whose strengthening may be advantageous; see \cite{NR4}. This leads us to
propose to strengthen existing edges $e(v_i^{\ell}\rightarrow v_j^{\ell})$, for 
$\ell=1,2,\dots,L$, determined by the $(i,j)$th entry of $A^+$ that correspond to a 
largest entry of the Wilkinson perturbation $W^K_{N}$ projected onto the cone 
${\mathcal S_{A^+}}$. The following result motivates this approach.

\begin{proposition}
One has $\rho_K\leq\rho_{K+1}\leq \rho$, for $1\leq K< N-1$.
\end{proposition}

\begin{proof}
The proof follows by observing that $P_{-1}^K\leq P_{-1}^{K+1}$, for $1\leq K<N-1$, and 
that $P_{-1}\equiv P_{-1}^{N-1}$.
\end{proof}

Moreover, consider the vector $\h^K_{\rm in}\in\R^N$ whose $i$th entry is the harmonic 
$K_{\rm in}$-centrality $\sum_{h\ne i} 1/p^K_{hi}$ of $v_i$ and the vector 
$\h^K_{\rm out}\in\R^N$ whose $i$th entry is its harmonic $K_{\rm out}$-centrality 
$\sum_{h\ne i} 1/p^K_{ih}$; see  \cite{NR4}. The following propositions can be shown 
similarly as Propositions \ref{gl_harm}, \ref{rho_harm}, and \ref{rho_gl}.

\begin{proposition}
For the multiplex global $K$-efficiency and the harmonic $K_{\rm in}$- and 
$K_{\rm out}$-centralities, the following equalities hold
$$
e_{\mathcal{A}}^K(\gamma)=\frac{1}{N(N-1)}\|\h^K_{\rm in}\|_1=
\frac{1}{N(N-1)}\|\h^K_{\rm out}\|_1,  \quad 1\leq K\leq N-1.
$$
\end{proposition}

\begin{proposition}
For the Perron root of the multiplex $K$-efficiency matrix and the harmonic $K_{\rm in}$- 
and $K_{\rm out}$-centralities one has, for $1\leq K\leq N-1$,
$$
\rho_K\leq\min\{ \|\h^K_{\rm in}\|_{\infty}, \|\h^K_{\rm out}\|_{\infty}\}.
$$
\end{proposition}

\begin{proposition}
For the Perron root of the multiplex $K$-efficiency matrix and the multiplex global 
$K$-efficiency one has, for $1\leq K\leq N-1$,
$$
\rho_K\leq N(N-1)e^K_{\mathcal{A}}(\gamma).
$$
\end{proposition}

The multiplex global $K$-efficiency is expected to increase the most by increasing the 
edge-weights that make the Perron root $\rho_K$ increase the most. In fact, when analyzing
$P^{K}_{-1}$, for $1\leq K\leq N-1$, we prefer to consider the eigenvector centrality over
the degree, that is to say, the $i$th entry of $\mathbf{y}_K$ instead of the in-degree of 
vertex $v_i$ (its harmonic $K_{\rm in}$-centrality) and the $j$th entry of $\mathbf{x}_K$ 
instead of the out-degree of vertex  $v_j$ (its harmonic $K_{\rm out}$-centrality); see, e.g.,
\cite{Bo,Es,Ne} for discussions on eigenvector centrality.

\section{The popularity approach}\label{sec4} 
This section considers the multiplex Perron communicability and discusses how it can be
used to determine the importance of an edge. The sensitivity of this measure to 
perturbations of the weights is studied. The multiplex Perron communicability has
prviously been described in \cite{EHNR}. It is the aim of the present paper to
compare the performance of the techniques of this section and Section \ref{sec3}.

The evaluation of the matrix $\exp_0(B)$ is very time-consuming when the supra-adjacency 
matrix $B$ is large. We therefore are interested in estimating the multiplex total 
communicability without calculating $\exp_0(B)$ by using the Perron root $\rho$ of 
$B=B(\gamma)$, the associated Perron vectors $\mathbf{x}$ and $\mathbf{y}$, and the 
Wilkinson matrix $W_{NL}=\mathbf{y}\mathbf{x}^T$. We propose to approximate 
$tc_B(\gamma)=\one_{NL}^T\exp_0(B)\one_{NL}$ by means of the 
\emph{multiplex Perron communicability}
\[
Pc_B(\gamma)=\exp_0({\rho})\,\one_{NL}^T W_{NL}\one_{NL},
\]
which is much easier to compute; see \cite{DLCJNR,EHNR} for related discussions.
The following result holds.

\begin{proposition}[\cite{DLCJNR}]\label{t_appr}
If the Perron root $\rho$ of $B$ is significantly larger than the magnitude of
the other eigenvalues of $B$, then
\[
tc_B(\gamma)\approx\kappa(\rho)Pc_B(\gamma).
\]
\end{proposition}

Thus, $tc_B(\gamma)$ depends on $Pc_B(\gamma)$ and the conditioning of the Perron root 
$\rho$. In the special case of an undirected network, the Perron vectors $\x$ and $\y$
coincide and, therefore, $\kappa(\rho)=1/\y^T\x=1$. Under the assumption of Proposition 
\ref{t_appr}, we then obtain that
\[
tc_B(\gamma)\approx Pc_B(\gamma).
\]
Additionally, since $\one_{NL}^T W_{NL}\one_{NL}=\|\mathbf{x}\|_1 \|\mathbf{y}\|_1$ and, 
$\forall \mathbf{z}\in \mathbb{C}^n$,
one has $\|\mathbf{z}\|_2\leq \|\mathbf{z}\|_1\leq \sqrt{n}\|\mathbf{z}\|_2$, the 
following bounds for the multiplex Perron communicability hold.

\begin{proposition}[\cite{DLCJNR}]\label{bou_p}
\[
\exp_0({\rho})\leq Pc_B(\gamma)\leq NL\exp_0({\rho}).
\]
\end{proposition}

Typically, $\exp_0({\rho})\gg NL$. Therefore, it suffices to consider $\exp_0({\rho})$ to
determine whether the multiplex Perron communicability is large or small.

\subsection{Sensitivity of multiplex total communicability}
We would like to identify the intra-layer edges that should be strengthened to enhance the
multiplex total communicability the most. The canonical way to identify the entries of the
block-diagonal portion of the supra-adjacency matrix \eqref{Bgamma}, whose weights should 
be increased, is to evaluate the Fr\'echet derivative $L_{\exp_0}(B,E)\in\R^{NL\times NL}$ 
at $B=B(\gamma)$ in the direction $E=\mathbf{e}_i\mathbf{e}_j^T\in\R^{NL\times NL}$ for 
$1\leq i,j\leq NL$. The Fr\'echet derivative $L_f(B,E)$ of a function $f$ at the matrix
$B$ in the direction $E$ is defined as
\[
f(B+E)=f(B)+L_f(B,E)+o(\|E\|_2)\mbox{~~~as~~~} \|E\|_2\rightarrow 0;
\]
see, e.g., \cite{Hi,NR4,Sc}. We are interested in determining intra-layer edges that have 
large weights, whose modification results in a relatively large change in the total 
communicability. Note that the sensitivity in the direction $\mathbf{e}_i\mathbf{e}_j^T$, 
i.e., $\mathbf{1}_{NL}^TL_{\exp_0}(B,\mathbf{e}_i\mathbf{e}_j^T)\mathbf{1}_{NL}$, is 
$\mathbf{e}_i^TL_{\exp_0}(B^T,\mathbf{1}_{NL}\mathbf{1}_{NL}^T)\mathbf{e}_j$; see 
\cite{Sc}. However, the evaluation of $L_{\exp_0}(B^T,\mathbf{1}_{NL}\mathbf{1}_{NL}^T)$ 
is very demanding (about $8$ times more arithmetic floating point operations than the 
evaluation of $\exp_0(B)$). 

We remark that one could approximate the gradient of $tc_{B}(\gamma)$ by using Arnoldi or 
Lanczos decompositions, as proposed for large-scale single-layer networks in 
\cite{Sc,NR4}. In the following subsection, we focus on another approach, that takes into
account the multiplex Perron communicability. The computations required are quite 
straightforward and not very demanding also for large-scale problems.

\subsection{Increasing the multiplex communicability}
We propose to determine the intra-layer edges with large weights whose modification yields
a relatively large change in the Perron root $\rho$ of the supra-adjacency matrix $B$. 
These intra-layer edges should be strengthened to enhance the multiplex Perron 
communicability the most. We will illustrate that modifications of the weights of the 
intra-layer edges identified by this technique give a relatively large change in the 
multiplex total communicability. 

Following \cite{NR4}, we construct an ``importance vector'' by multiplying the positive 
entries of 
\[
B_d={\rm blkdiag}[A^{(1)},\dots,A^{(L)}]
\]
element by element by the corresponding entries of $W_{NL}|_{\mathcal{S}_{B_d}}$, where 
$\mathcal{S}_{B_d} \subseteq\R^{NL \times NL}$ is the cone of all nonnegative matrices 
with the same sparsity structure as $B_d$, and then choose the weights 
$a_{ij}^{(\ell)}\in\mathcal{A}$ that correspond to the largest entries of this vector. 
Since supra-adjacency matrices typically are quite large, one generally computes their 
right and left Perron vectors by an iterative method that only requires the evaluation of
matrix-vector products with the matrix $B_d$ and its transpose. Clearly, one does not have
to store $B_d$, but only $\mathcal{A}$, to evaluate matrix-vector products with the matrix
$B_d$ and its transpose.

To estimate the potential for increase in communicability, we propose to also evaluate an
approximation of the \emph{structured multiplex Perron communicability}, which is defined 
by
\[
Pc^{\rm struct}_B(\gamma)=\exp_0({\rho})\,\one_{NL}^T W_{NL}|_{\mathcal{S}_{B_d}} 
\one_{NL}.
\]
The following result holds.

\begin{proposition}
\[
Pc^{\rm struct}_B(\gamma)\leq Pc_B(\gamma).
\]
\end{proposition}
\begin{proof}
The proof follows from the inequality $W_{NL}|_{\mathcal{S}_{B_d}}\leq W_{NL}$. 
\end{proof}

Due to Proposition \ref{bou_p}, we have $Pc^{\rm struct}_B(\gamma)\leq NL\exp_0({\rho})$. 
This bound can be refined as the following result shows.

\begin{proposition}
\[
Pc^{\rm struct}_B(\gamma)\leq NL\exp_0({\rho}) \frac{\kappa^{\rm struct}(\rho)}{\kappa(\rho)}.
\]
\end{proposition}

\begin{proof}
Since $\one_{NL}^T W_{NL}|_{\mathcal{S}_{B_d}}\one_{NL}=
\|\mathrm {vec}(W_{NL}|_{\mathcal{S}_{B_d}})\|_1\leq NL\|W_{NL}|_{\mathcal{S}_{B_d}}\|_F$, 
we have the upper bound 
$Pc^{\rm struct}_B(\gamma)\leq NL\exp_0({\rho})\| W_{NL}|_{\mathcal{S}_{B_d}}\|_F$. The 
proof follows by observing that 
$\kappa^{\rm struct}(\rho)=\kappa(\rho)\| W_{NL}|_{\mathcal{S}_{B_d}}\|_F$; cf. eq. \eqref{kst}.
 \end{proof}

\section{Numerical tests}\label{sec5} 
The numerical tests reported in this section have been carried out using MATLAB R2024a on 
a $3.2$ GHz Intel Core i7 6 core iMac. The Perron root, and the left and right Perron 
vectors for small to moderately sized networks can easily be evaluated by using the MATLAB 
function {\sf eig}. For large-scale multiplexes, these quantities can be computed by the 
MATLAB function {\sf eigs} or by an Arnoldi algorithm (one-sided or two sided).

\subsection{Single-layers networks}
In the simple case when $L=1$, the supra-adjacency matrix $B$, the third-order tensor 
${\mathcal A}$, as well as $B_d$ and $A^+$, reduce to the adjacency matrix 
$A\in\R^{N\times N}$ for the given single-layer network. Also, 
${\mathcal S_{A^+}}\equiv{\mathcal S_{B_d}}\equiv{\mathcal S_{A}}$. Nevertheless, a few 
comments about the well-known {\em Air500} and {\em Autobahn} data sets may be of interest
to a reader since these networks allow simple illustrations of the concepts of efficiency 
and popularity.

\begin{example}[Air500 data set]
Consider the adjacency matrix $A\in\R^{500\times 500}$ for the network {\em Air500} 
in \cite{air500_autobahn}. This data set describes flight connections for the top 500 
airports worldwide based on total passenger volume. The flight connections between 
airports are for the year from 1 July 2007 to 30 June 2008. The network is represented by
a directed unweighted connected graph ${\mathcal{G}}$ with $N=500$ vertices and $24009$ 
directed edges. The vertices of the network are the airports and the edges represent 
direct flight routes between two airports. 

The global efficiency is $e_A =e_A^{499}\equiv e_A^5=4.8392\cdot 10^{-1}$, see 
\cite[Example 5]{NR3}, and the total communicability is $tc_{A}=1.9164\cdot 10^{38}$.
The Perron communicability is $Pc_{A}=1.9132\cdot 10^{38}$. The flight connection from 
the Frankfurt FRA Airport (vertex $v_{161}$) to the New York JFK Airport (vertex 
$v_{224}$) is more efficient than the flight connection from New York JFK Airport (vertex
$v_{224}$) to the Atlanta ATL Airport (vertex $v_{24}$), i.e., strengthening the former 
edge has a larger impact on the global efficiency than strengthening the latter edge. The 
latter edge is the most popular edge, i.e., increasing the weight for this edge increases 
the total communicability the most. The former edge appears in several shortest 
paths that connect airports, while the latter picks up travelers from several major 
airports and transmits them to several major airports. Note that the information provided 
by the efficiency matrix is the same as the one given by $P^{2}_{-1}$, so that
the perturbation that increases the global $2$-efficiency the most also increases the 
global efficiency the most; cf. Table \ref{tab1}.

\begin{table}[ht]
\centering
\begin{tabular}{c*{2}{|c}}
$K$           & $(h,k)$ & $e_{{A}}^K(1)$  \\
\hline
5          & ${(161,224)}$ & $4.8392\cdot 10^{-1}$     \\
4          & ${(161,224)}$ & $4.8387\cdot 10^{-1}$   \\
3          & ${(161,224)}$ & $4.7909 \cdot 10^{-1}$  \\
2          & ${(161,224)}$ & $3.6044\cdot 10^{-1}$   \\
1          & $(224,24)$ & $9.6228\cdot 10^{-2}$  \\
\end{tabular}
\caption{Air500 data set. Indices chosen by the procedure are shown in the second column 
and the global $K$-efficiency is displayed in the third column for $K=1,2,\ldots,5$.} 
\label{tab1}
\end{table}
\end{example}

\begin{example}[Autobahn data set]
Consider the undirected unweighted graph $\mathcal G$ that represents the German 
highway system network {\em Autobahn}. The graph, which is available at 
\cite{air500_autobahn}, has $N=1168$ vertices representing German locations and $1243$ 
edges representing highway segments that connect them. Therefore, the adjacency matrix
$A\in\R^{1168\times 1168}$ for this network has $2486$ nonvanishing entries. 

The global efficiency is $e_A=e_A^{1167}\equiv e_A^{62}=6.7175\cdot 10^{-2}$; see
\cite[Example 6]{NR3}. The total communicability is $tc_{A}=1.2563\cdot 10^{4}$ and the
Perron communicability is $Pc_{A}=2.2448\cdot 10^{3}$. The highway segment that connects 
Duisburg (vertex $v_{219}$) and  Krefeld (vertex $v_{565}$) turns out to be more efficient
than the highway segment that connects Duisburg (vertex $v_{219}$) and D\"usseldorf 
(vertex $v_{217}$), which instead turned out to be the most popular edge. Note that the 
information provided by the efficiency matrix is the same as the one given by
$P^{4}_{-1}$, so that the perturbation that increases the global $4$-efficiency the most 
also increases the global efficiency the most; cf. Table \ref{tab2}.

\begin{table}[ht]
\centering
\begin{tabular}{c*{2}{|c}}
$K$           & $(h,k)$ & $e_{{A}}^K(1)$  \\
\hline
62          & ${(565,219)}$ & $6.7175 \cdot 10^{-2}$     \\
\vdots  & \vdots& \vdots\\
5          & ${(565,219)}$ & $7.9991\cdot 10^{-3}$   \\
4          & ${(565,219)}$ & $6.1823\cdot 10^{-3}$   \\
3          & ${(219,217)}$ & $4.6017  \cdot 10^{-3}$  \\
2          & ${ (693,543)}$ & $3.2082 \cdot 10^{-3}$   \\
1          & $(219, 217)$ & $1.8238\cdot 10^{-3}$  \\
\end{tabular}

\caption{Autobahn data set. The second column shows indices chosen by the procedure for
$K=1,2,\ldots,5$ and $K=62$, and the third column displays the global $K$-efficiency.} 
\label{tab2}
\end{table}
 \end{example}

\subsection{Multiplex networks}
The computations in this subsection use the two-sided Arnoldi method described in 
\cite{ZH}.

\begin{example}[European airlines data set]
We consider the undirected, unweighted, and connected network consisting of $N=417$ 
vertices that represent European airports and  $L=37$ layers that represent different 
airlines operating in Europe. Each edge represents a flight between airports. The network
can be downloaded from \cite{B_repository}. 
In this multiplex all the maximal shortest paths are made up of 7 intra-layer edges 
and 2 layer switches, with  $\gamma=1$ that reflects the effort required to change 
airlines for connecting flights. Thus, for any $K\geq7$, one has 
$e_{\mathcal{A}}(1)=e_{\mathcal{A}}^{K}(1)$ with
$$
e_{\mathcal{A}}(1)=\frac{1}{N(N-1)}\sum_{i,j\ne i} \frac{1}{p_{ij}}= 0.3477.
$$
For all $K\geq3$, the largest entries of $W_{N}^K|_{\mathcal S_{A^+}}$ correspond to the 
route between  vertices $v_{40}$ and $v_{15}$. Thus, the information provided by the 
multiplex efficiency matrix $P_{-1}$ is the same as the information given by $P^{3}_{-1}$, 
so that the perturbation that increases the multiplex global $3$-efficiency the most is the 
same that increases the multiplex global efficiency the most; cf. Table \ref{tab3}. This indicates 
that the Perron root $\rho$ of $P_{-1}$ may be increased the most by doubling the number
of flights between the Barcelona (vertex $v_{40}$) and Amsterdam  (vertex $v_{15}$) airports. 
The airlines that operate this route are EasyJet (layer $3$), KLM (layer $9$), Vueling (layer $21$), 
and Transavia Holland (layer $27$), all with the same number of flights.  
\begin{table}[ht]
\centering
\begin{tabular}{c*{2}{|c}}
$K$           & $(h,k)$ & $e_{\mathcal{A}}^K(1)$  \\
\hline
7          & ${ (40,15)}$ &$3.4766\cdot 10^{-1}$     \\
6          & ${(40,15)}$ & $3.4766\cdot 10^{-1}$      \\
5          & ${(40,15)}$ & $3.4746\cdot 10^{-1}$     \\
4          & ${ (40,15)}$ & $3.4416\cdot 10^{-1}$   \\
3          & ${ (40,15)}$ & $3.1949 \cdot 10^{-1}$  \\
2          & ${(38,15)}$ & $1.8393\cdot 10^{-1}$   \\
1          & $(38,15)$ & $3.4046\cdot 10^{-2}$  \\
\end{tabular}

\caption{European airlines data set. Indices chosen by the procedure in the second column 
and multiplex global $K$-efficiency in the third column for $K=1,2,\ldots,7$.} 
\label{tab3}
\end{table}

Increasing the number of flights of each airline by 25 percent, so as to increase by one the total 
number of flights on the route, yields the perturbed third-order adjacency tensor 
${\widetilde {\mathcal{A}}}$ and the global efficiency 
$$
e_{\widetilde {\mathcal{A}}}(1)=\frac{1}{N(N-1)}
\sum_{i,j\ne i} \frac{1}{{\tilde  p}_{ij}}= 0.3480.
$$
\end{example}

As for the popularity approach, one has 
$$tc_B(1)=\one_{NL}^T \exp_0(B)\one_{NL}=2.4930\cdot 10^{20}.$$
The Perron root $\rho=38.3714$ of $B$ (with $\exp_0({\rho})=4.6183\cdot 10^{16}$, 
$\kappa(\rho)=1$ and $\kappa^{\rm struct}(\rho)=5.3310\cdot 10^{-2}$) is significantly 
larger than the other eigenvalues. We obtain
\begin{eqnarray*}
Pc_B(1)&=&\exp_0({\rho})\,\one_{NL}^T W_{NL}\one_{NL}=1.9637\cdot 10^{20},\\
Pc_B^{\rm struct}(1)&=&\exp_0({\rho})\,\one_{NL}^T W_{NL}|_{\mathcal{S}_{B_d}}
\one_{NL}=1.4733\cdot 10^{17}.
\end{eqnarray*}
The largest entries of $W_{NL}|_{\mathcal{S}_{B_d}}$ correspond to edge 
$e(v_{38}^{1}\leftrightarrow v_{2}^{1})$. This indicates that the Perron root may be 
increased the most by doubling the number of flights operated by Lufthansa airline between
the Munich (vertex $v_{38}$) and Frankfurt (vertex $v_2$) airports. Note that Lufthansa 
(layer $1$) is the only operating company for this route. For the perturbed supra-adjacency 
matrix ${\widehat B}$ one has the Perron root ${\hat \rho}=38.3798$,  with 
$\exp_0({{\hat \rho}})=4.6572\cdot 10^{16}$, and
$$tc_{\widehat B}(1)=\one_{NL}^T \exp_0(\widehat B)\one_{NL}=2.5056\cdot 10^{20}.$$

Finally, we observe that
\[
e_{\widehat {\mathcal{A}}}(1)=0.3479<0.3480= e_{{\widetilde {\mathcal{A}}}}(1);\qquad 
tc_{{\widetilde B}}(1)=2.4972\cdot 10^{20}<2.5056\cdot 10^{20}=tc_{\widehat {B}}(1).
\]

\begin{example}[London transportation data set]
Consider the undirected, weighted,  and connected network consisting of $N=369$ vertices 
that represent train stations in London and  $L=3$ layers that represent the networks of 
stations connected by

1. Tube - All underground  lines (e.g., District, Circle, etc) aggregated; 

2.  Overground;

3. Docklands Light Railway (DLR).

\noindent Each intra-layer edge represents a route between stations.  Data was collected in 
2013. The network can be downloaded from \cite{D}.

In this multiplex all the maximal shortest paths are made up of $40$ intra-layer edges and
$2$ layer switches, with  $\gamma=1$. 

Thus, for all {$K\geq 40$}, one has $e_{\mathcal{A}}^{K}(1)=e_{\mathcal{A}}(1)$, with
$$
e_{\mathcal{A}}(1)=\frac{1}{N(N-1)}\one_{N}^T P_{-1}\one_N= 0.1126.
$$
Also, for all { $K\geq10$},   the largest entries of $W_{N}^K|_{\mathcal S_{A^+}}$ correspond 
to the route between  vertices $v_{185}$ and $v_{182}$. This indicates that the Perron root  of 
the multiplex efficiency matrix, ${\rho}(P_{-1})=46.5551$, may be increased the most by adding 
a new underground  line to the $3$  lines operating  between the Euston Square ($v_{185}$) 
and King's Cross St. Pancras ($v_{182}$).  

\begin{table}[ht]
\centering
\begin{tabular}{c*{2}{|c}}
$K$           & $(h,k)$ & $e_{\mathcal{A}}^K(1)$  \\
\hline
40        & ${ (185,182)}$ & $1.1261\cdot 10^{-1}$      \\
 \vdots         & { \vdots}  &  \vdots   \\
10         & ${ (185,182)}$ & $6.4761\cdot 10^{-2}$   \\
9          & ${ (182,39)}$ & $5.8334\cdot 10^{-2}$  \\
8          & ${ (182,39)}$ & $5.1699\cdot 10^{-2}$  \\
7         & ${ (182,39)}$ & $4.4932\cdot 10^{-2}$  \\
6         & ${ (182,39)}$ & $3.7978\cdot 10^{-2}$  \\
5         & ${ (182,39)}$ & $3.1098\cdot 10^{-2}$  \\
4         & ${(185,182)}$ & $2.4577\cdot 10^{-2}$  \\
3         & ${(185,182)}$ & $1.8541\cdot 10^{-2}$  \\
 2         & ${ (182,39)}$ & $1.2894\cdot 10^{-2}$  \\
1          & ${ (182,39)}$ & $7.2464\cdot 10^{-3}$  \\
\end{tabular}

\caption{London transportation data set. The second column shows indices chosen  by the procedure for
$K=1,2,\ldots,10$ and $K=40$,  and the third column displays the multiplex global 
$K$-efficiency.} 
\end{table}

As for the popularity approach, one has $$tc_B(1)=\one_{NL}^T \exp_0(B)\one_{NL}=5.6238\cdot 10^{4}.$$

The Perron root $\rho=6.5138$ of $B$ (with $\exp_0({\rho})=6.7341\cdot 10^{2}$, 
$\kappa(\rho)=1$ and $\kappa^{\rm struct}(\rho)=5.0028\cdot 10^{-1}$) is significantly larger than the other 
eigenvalues. $$Pc_B(1)=\exp_0({\rho})\,\one_{NL}^T W_{NL}\one_{NL}=2.0831\cdot 10^{4}$$
$$Pc_B^{\rm struct}(1)=\exp_0({\rho})\,\one_{NL}^T W_{NL}|_{\mathcal{S}_{B_d}}\one_{NL}=
1.6214\cdot 10^{3}.$$ 
The largest entries of $W_{NL}|_{\mathcal{S}_{B_d}}$ correspond to the edge 
$e(v_{182}^{1}\leftrightarrow v^{1}_{39})$ in layer $1$. This indicates that the Perron root 
may be increased the most by adding a new underground line to the $3$ lines operating 
between King's Cross St. Pancras ($v_{182}$) and Farrington Station ($v_{39}$).  

Let us take into account both the efficiency and the popularity approaches, by adding the 
underground line Euston Square - King's Cross St. Pancras - Farrington Station. As for the 
perturbed multiplex, one has the Perron root $\rho({\widehat P_{-1}})=46.9491$ and
$$
e_{\widehat {\mathcal{A}}}(1)= \frac{1}{N(N-1)}\one_{N}^T \widehat P_{-1}\one_{N}= 0.1132,
$$
whereas, for the perturbed supra-adjacency matrix ${\widehat B}$, one has the Perron root  
${\hat \rho}=7.4155$,  $\exp_0(\hat \rho)=1.6606\cdot 10^{3}$, and
$$tc_{\widehat B}(1)=\one_{NL}^T \exp_0(\widehat B)\one_{NL}=7.0644\cdot 10^{4}.$$

\end{example}

In the above examples, we used the Wilkinson perturbation
associated with the Perron root of $B$ and projected onto ${\mathcal{S}_{B_d}}$ (or associated 
with the Perron root of the $P_{-1}^K$, for a suitable $K$, and projected onto ${\mathcal{S}_{A^+}}$)
to determine the edge, such that a change in its weight has the largest effect on the 
total communicability (or the global efficiency) of the multiplex. If we are interested in identifying 
more than one edge, whose weight should be changed to increase such measure of communication, 
then we can either repeat the computations with the network obtained by having modified one 
edge-weight or consider the edge that is identified by the second largest entry of the same 
matrix.

\section{Concluding remarks}\label{sec6} 
Two measures of communicability in multiplex networks are considered: 
multiplex global efficiency and multiplex total communicability. Their sensitivity to 
changes in edge-weights is investigated. Approximations of both measures can be evaluated
also for large multiplex networks. They measure different aspects of communicability and 
shed light on which edges should be strengthened (e.g., which roads should be widened) to 
increase the communicability of the network the most. Application of these concept to
single-layer and multiplex transportation networks are presented.

%
%
%
%


\begin{thebibliography}{99}

\bibitem{BBV}
A. Barrat, M. Barthelemy, and A. Vespignani, Dynamical Processes on Complex Networks,
Cambridge University Press, Oxford, 2008.

\bibitem{BK}
M. Benzi and C. Klymko, Total communicability as a centrality measure, J. Complex Netw., 1
(2013), pp. 124–149.

\bibitem{B_repository}
K. Bergermann, Multiplex-matrix-function-centralities,\\
\texttt{https://github.com/KBergermann/Multiplex-matrix-function-centralities}.

\bibitem{BS1} 
K. Bergermann and M. Stoll, Fast computation of matrix function-based centrality 
measures for layer-coupled multiplex networks, Phys. Rev. E, 105 (2022), Art. 034305.

\bibitem{BS} 
K. Bergermann and M. Stoll, Orientations and matrix function-based centralities in 
multiplex network analysis of urban public transport, Appl. Netw. Sci., 6 (2021), pp. 
1--33.

\bibitem{Bo}
P. F. Bonacich, Power and centrality: A family of measures, Am. J. Sociol., 92 (1987), 
pp. 1170--1182.
,
\bibitem{BMRZ}
C. Brezinski, G. Meurant, and M. Redivo-Zaglia, A Journey through the History of Numerical 
Linear Algebra, SIAM, Philadelphia, 2022.

\bibitem{BRZ}
C. Brezinski and M. Redivo-Zaglia, Rational extrapolation for the PageRank vector,
Math. Comp., 77 (2008), pp. 1585--1898.

\bibitem{BRZ2}
C. Brezinski and M. Redivo-Zaglia, The genesis and early developments of Aitken’s 
process, Shanks’ transformation, the $\varepsilon$–algorithm, and related fixed point 
methods, Numer Algorithms, 80 (2019), pp. 11--133.

\bibitem{BRZ3}
C. Brezinski and M. Redivo-Zaglia, Reuben Louis Rosenberg (1909–1986) and the 
Stein-Rosenberg theorem, Electron. Trans. Numer. Anal., 58 (2023), pp. A1--A38.

\bibitem{BRZRS}
C. Brezinski, M. Redivo-Zaglia, G. Rodriguez, and S. Seatzu, Multi-parameter 
regularization techniques for ill-conditioned linear systems, Numer. Math., 94 (2003), 
pp. 203--228.

\bibitem{BRZS0}
C. Brezinski, M. Redivo-Zaglia, and Y. Saad, Shanks sequence transformations and Anderson 
acceleration, SIAM Rev., 60 (2018), pp. 646--669.

\bibitem{BRZS}
C. Brezinski, M. Redivo-Zaglia, and H. Sadok, New look-ahead Lanczos-type algorithms for 
linear systems, Numer. Math., 83 (1999), pp. 53--85.

\bibitem{CRZT}
S. Cipolla, M. Redivo‐Zaglia, F. Tudisco, Shifted and extrapolated power methods for 
tensor $\ell^p$-eigenpairs, Electron. Trans. Numer. Anal., 53 (2020), pp. 1--27.

\bibitem{D}
M. De Domenico, 
\texttt{https://manliodedomenico.com/data.php}.

\bibitem{DSGA}
M. De Domenico, A. Sol\'e-Ribalta, S. Gómez, and A. Arenas, Navigability of interconnected 
networks under random failures, PNAS, 111 (2014), No. 23,  pp. 8351--8356.

\bibitem{DSOGA}
M. De Domenico, A. Sol\'e-Ribalta, E. Omodei, S. Gómez, and A. Arenas, Centrality in 
interconnected multilayer networks, arXiv:1311.2906v1, (2013).

\bibitem{DLCJNR}
O. De la Cruz Cabrera, J. Jin, S. Noschese, and L. Reichel, Communication in complex 
networks, Appl. Numer. Math., 172 (2022), pp. 186--205.

\bibitem{air500_autobahn}
Dynamic Connectome Lab - Data Sets. 
\url{https://sites.google.com/view/dynamicconnectomelab}

\bibitem{EHNR}
S. El-Halouy, S. Noschese, and L. Reichel, Perron communicability and sensitivity of 
multilayer networks, Numer. Algorithms, 92 (2023), pp. 597--617.

\bibitem{Es}
E. Estrada, The Structure of Complex Networks: Theory and Applications, Oxford University
Press, Oxford, 2011.

\bibitem{ER}
E. Estrada and J. A. Rodriguez-Velazquez, Subgraph centrality in complex networks, Phys. 
Rev. E, 71 (2005), Art. 056103.

\bibitem{Hi}
N. J. Higham, Functions of Matrices: Theory and Computation, SIAM, Philadelphia, 2008.

\bibitem{KAB}
M. Kivil\"a, A. Arenas, M. Barthelemy, J. P. Gleeson, Y. Moreno, and M. A. Porter,
Multilayer networks, J. Complex Netw., 2 (2014), pp. 203--271.

\bibitem{L}
G. L. Litvinov, Maslov dequantization, idempotent and tropical mathematics: A brief 
introduction, J. Math. Sci., 140 (2007), pp. 426--444.


\bibitem{Ne}
M. E. J. Newman, Networks: An Introduction, Oxford University Press, Oxford, 2010.

\bibitem{NP}
S. Noschese and L. Pasquini, Eigenvalue condition numbers: Zero-structured versus 
traditional, J. Comput. Appl. Math., 185  (2006), pp. 174--189. 

\bibitem{NR2}
S. Noschese and L. Reichel, Network analysis with the aid of the path length matrix, 
Numer. Algorithms, 95 (2024) pp. 451--470.

\bibitem{NR3}
S. Noschese and L. Reichel, Enhancing multiplex global efficiency, 
Numer. Algorithms, 96 (2024) pp. 397--416.

\bibitem{NR4}
S. Noschese and L. Reichel, Edge importance in complex networks, Numer. Algorithms, in
press.

\bibitem{Sc}
M. Schweitzer, Sensitivity of matrix function based network communicability measures: 
Computational methods and a priori bounds, SIAM J. Matrix Anal. Appl., 44 (2023), pp. 
1321--1348.

\bibitem{Wi}
J. H. Wilkinson, Sensitivity of eigenvalues II, Util. Math., 30 (1986), pp. 
243--286.

\bibitem{ZH}
I. N. Zwaan and M. E. Hochstenbach, Krylov–Schur-type restarts for the two-sided Arnoldi
method, SIAM J. Matrix Anal. Appl., 38 (2017), pp. 297--321.

\end{thebibliography}
\end{document}